\documentclass[12pt]{amsart}
\pdfoutput=1
\usepackage[margin=1in]{geometry}
\usepackage[utf8]{inputenc}
\usepackage{graphicx, cite}
\usepackage{amsmath,amssymb}
\usepackage{amsthm}
\usepackage{color}
\usepackage{tikz}
\usepackage[utf8]{inputenc}
\usepackage{pgfplots}
\pgfplotsset{compat=1.18}
\usetikzlibrary{fillbetween}
\allowdisplaybreaks

\makeatletter
\def\l@section{\@tocline{1}{12pt plus2pt}{0pt}{}{\bfseries}}
\def\l@subsection{\@tocline{2}{0pt}{2pc}{2pc}{}}
\makeatother
\makeatletter
\def\subsection{\@startsection{subsection}{2}{\z@}%
	{-3.25ex\@plus -1ex \@minus -.2ex}%
	{1.5ex \@plus .2ex}%
	{\normalfont\bfseries\boldmath}}
\def\subsubsection{\@startsection{subsubsection}{3}%
	\z@{.5\linespacing\@plus.7\linespacing}{-.5em}%
	{\normalfont\bfseries\boldmath}}
\renewcommand\paragraph{\@startsection{paragraph}{4}{\z@}%
	{3.25ex \@plus1ex \@minus.2ex}%
	{-1em}%
	{\normalfont\normalsize\bfseries}}
\makeatother

\theoremstyle{plain}
\newtheorem{thm}{Theorem}[section]

\newtheorem{lem}[thm]{Lemma}

\theoremstyle{definition}

\theoremstyle{remark}
\newtheorem{rem}[thm]{Remark}

\theoremstyle{plain}

\numberwithin{equation}{section}

\theoremstyle{plain} 
\newcommand{\thistheoremname}{}
\newtheorem{genericthm}[thm]{\thistheoremname}

  \newtheorem*{genericthm*}{\thistheoremname}
\newenvironment{namedthm*}[1]
  {\renewcommand{\thistheoremname}{#1}%
   \begin{genericthm*}}
  {\end{genericthm*}}

\newcommand{\HH}{{\mathbb H}}
\newcommand{\R}{{\mathbb R}}

\newcommand{\Z}{{\mathbb Z}}

\newcommand{\calS}{{\mathcal S}}
\newcommand{\supp}{{\textnormal{supp}}}

\makeatletter
\newcommand{\vast}{\bBigg@{4}}
\newcommand{\Vast}{\bBigg@{5}}
\makeatother

\def\udot#1{\ifmmode\oalign{$#1$\crcr\hidewidth.\hidewidth
    }\else\oalign{#1\crcr\hidewidth.\hidewidth}\fi}

\def\R{\mathbb{R}}
\def\Z{\mathbb{Z}}

\def\beq{\begin{equation}}
\def\eeq{\end{equation}}
\makeatletter
\newcommand{\doublewidetilde}[1]{{%
  \mathpalette\double@widetilde{#1}%
}}
\newcommand{\double@widetilde}[2]{%
  \sbox\z@{$\m@th#1\widetilde{#2}$}%
  \ht\z@=.9\ht\z@
  \widetilde{\box\z@}%
}
\makeatother

\usepackage{tikz}

\makeatletter
\def\@makefnmark{%
  \leavevmode
  \raise.9ex\hbox{\fontsize\sf@size\z@\normalfont\tiny\@thefnmark}}
\makeatother

\begin{document}
	
\title[]{Sparse domination and the bi-parameter Hilbert transform}

\author{Bingyang Hu}
\address{(Bingyang Hu) Department of Mathematics and Statistics\\
        Auburn University\\
        Auburn, Alabama, U.S.A, 36849}
\email{bzh0108@auburn.edu}

\begin{abstract}
We study domination of the bi-parameter Hilbert transform
by $(1,1)$-type sparse forms over axis-parallel rectangles,
and show that such domination fails even for smooth test
functions compactly supported in $(0,1)^2$.
Our proof combines an explicit construction of two finite
lattices based on the arithmetic of $\Z[\sqrt2]$
with a randomization argument.
\end{abstract}

\date{October 3, 2026}


\maketitle

\section{Introduction}

Consider the \emph{bi-parameter Hilbert transform} in $\R^2$, which is defined as
\begin{align}\label{BPHTeq01}
 \HH f(x)
 &:=\frac{1}{\pi^2}\operatorname{p.v.}
 \int_{\R^2}\frac{f(y)}{(x_1-y_1)(x_2-y_2)}\,dy \nonumber\\
 &:=\frac{1}{\pi^2} \lim_{\varepsilon_1,\varepsilon_2\to0^+}
 \int_{\substack{|x_1-y_1|>\varepsilon_1\\
|x_2-y_2|>\varepsilon_2}}
 \frac{f(y)}{(x_1-y_1)(x_2-y_2)}\,dy, 
\end{align}
for $f\in C_c^\infty(\R^2)$. The \emph{goal} of this paper is to study sparse bounds for $\HH$.
We call $R\subset\R^2$ a rectangle if
$R=[a,b)\times[c,d)$ for real numbers $a<b$ and $c<d$, and a finite or countable collection $\calS$ of rectangles is said to be
\emph{a sparse family of rectangles} if there are pairwise disjoint measurable sets
$E_R\subset R$ satisfying $|E_R|\ge |R|/2$ for every $R\in\calS$.
We write
$$
 \langle f \rangle_R:=\frac{1}{|R|}\int_R f
$$
for the average of $f$ over $R$, and define the \emph{$(1, 1)$-type sparse form} by
\begin{equation} \label{20260928eq01}
 \Lambda_{\calS}(f,g)
 :=\sum_{R\in\calS}|R|
 \langle |f|\rangle_R\langle |g|\rangle_R.
\end{equation} 

Our main result is the following.

\begin{thm}\label{thmmain}
There exist absolute constants $c,C>0$ such that, for every sufficiently
large integer $n$, there are real-valued functions
$f_n,g_n\in C_c^\infty((0,1)^2)$ satisfying
\begin{equation} \label{mainest01}
 \langle\HH f_n,g_n\rangle\ge c\sqrt n
\end{equation} 
and
\begin{equation} \label{mainest02}
 \Lambda_{\calS}(f_n, g_n)\le C
 \qquad\text{for every sparse family} \ \mathcal S \ \textrm{of rectangles} \  \textrm{in} \ \R^2.
\end{equation}
Consequently, there is \emph{no} absolute constant $A>0$ such that,
for every pair $f,g\in C_c^\infty(\R^2)$, there exists a sparse
family of rectangles $\calS$ satisfying
\begin{equation} \label{sparseform}
 |\langle \HH f,g\rangle|\le A\Lambda_{\calS}(f,g).
\end{equation} 
\end{thm}

Sparse domination is an \emph{important} topic in harmonic analysis.
It provides a quantitative refinement of $L^p$ estimates for operators and has proved useful in obtaining weighted norm inequalities and endpoint estimates.
Much of this theory has been developed in the \emph{single}-parameter
setting, which dates back to the early work of Lerner \cite{Lerner2013a,Lerner2013b} established domination by positive
sparse operators in Banach function space norms and gave a
simple proof of the $A_2$ theorem. Subsequent developments
include both pointwise sparse bounds and estimates for the
forms associated with the operators. The literature in this direction is extensive, and we list several representative works for the reader's convenience.

\begin{enumerate}
\item[$\bullet$] Pointwise sparse bounds for Calder\'on--Zygmund
operators and positive dyadic shifts
(see, e.g., \cite{CondeAlonsoRey2016,Lacey2017,Lerner2016}).

\item[$\bullet$] Sparse form bounds for rough singular integrals
(see, e.g., \cite{CondeAlonsoCuliucDiPlinioOu2017}).

\item[$\bullet$] Sparse form bounds for multilinear singular
integrals, including the bilinear Hilbert transform
(see, e.g., \cite{CuliucDiPlinioOu2018}).

\item[$\bullet$] Sparse form bounds for singular Radon transforms
(see, e.g., \cite{CladekOu2018,Oberlin2019,Hu2020}).

\item[$\bullet$] Sparse form bounds for Fourier multipliers,
maximal operators, and square functions
(see, e.g.,
\cite{Lacey2019,LaceyMenaReguera2019,BeltranRoosSeeger2024}).

\item[$\bullet$] Applications of sparse and dyadic methods
in complex analysis
(see, e.g.,
\cite{AlemanPottReguera2017,RahmTchoundjaWick2017,
HuLiShiWick2021,GanHuKhan2022,HuHuoLanzaniPalenciaWagner2024}).

\item[$\bullet$] Generalizations to convex body domination
(see, e.g., \cite{NazarovPetermichlTreilVolberg2017}).
\end{enumerate}

In contrast, fewer results are available in the \emph{multi}-parameter
setting. An important result is due to Barron, Conde-Alonso, Ou,
and Rey \cite{BarronCondeAlonsoOuRey2019}, who proved that sparse domination of the form \eqref{sparseform} fails for the dyadic strong maximal function and cannot hold uniformly for bi-parameter
martingale transforms. More recently, Conde-Alonso, Lorist, and Rey
\cite{CondeAlonsoLoristRey2026} developed an elegant cancellative sparse
domination framework involving a maximal median operator,
which yields a pointwise bound for the cancellative dyadic
strong maximal operator.
For further developments in this direction, see
\cite{Rey2026,OmbrosiRey2026}.

\begin{rem}
Barron and Pipher \cite{BarronPipher2017} obtained sparse form
bounds involving square functions for bi-parameter martingale
transforms, cancellative dyadic shifts, and paraproduct-free
Journ\'e operators.
For bi-parameter martingale transforms, their bound takes
the form
\begin{equation} \label{squarefunctionsparse}
|\langle Tf,g\rangle|
\lesssim
\sum_{R\in\calS}|R|
\langle Sf\rangle_R\langle Sg\rangle_R,
\end{equation} 
where $S$ is the bi-parameter dyadic square function.
These estimates provide a square-function approach to sparse
domination in the bi-parameter setting.
Here, the square functions involve Haar coefficients over
rectangles not necessarily contained in $R$, so the averages
$\langle Sf\rangle_R$ and $\langle Sg\rangle_R$ may depend on
the values of $f$ and $g$ outside $R$. In this sense,
\eqref{squarefunctionsparse} lacks the locality required by the usual
sparse domination estimates.
\end{rem}

Returning to our main result, to the best of our knowledge,
Theorem~\ref{thmmain} provides the \emph{first} counterexample to the $(1, 1)$--type sparse form domination of the form \eqref{20260928eq01} 
for a classical bi-parameter singular integral operator. 

\vspace{0.1cm}

The \emph{key} ingredient of our construction is to make use of a pair of finite lattices obtained from the arithmetic of $\Z[\sqrt2]$ (see, Section \ref{20260928subsec01}). Their separation properties give the uniform estimate \eqref{mainest02} for every
sparse form (see, Section \ref{20261003subsecupper}), while each output point corresponds to sufficiently many
input points with large Hilbert kernel contributions (see, Lemma \ref{20261001lem04} and  \eqref{20261001eq23}).
To this end, a randomization argument yields the desired lower bound \eqref{mainest01}.
Finally, in Section~\ref{Sec04}, we show that the counterexample
constructed in Theorem~\ref{thmmain} does \emph{not} demonstrate
the failure of sparse domination of the form \eqref{sparseform}
for the strong maximal operator, as established by Barron,
Conde-Alonso, Ou, and Rey \cite{BarronCondeAlonsoOuRey2019}.

\begin{rem}
As suggested by Conde-Alonso, Lorist, and Rey
\cite{CondeAlonsoLoristRey2026}, an interesting question is
\emph{whether one can establish a suitable sparse bound for the
bi-parameter Hilbert transform $\HH$}.
In particular, it would be useful to obtain such a bound from which its $L^p$ and endpoint estimates follow as immediate consequences. The present paper does \emph{not} provide a positive result to thie question, but as pointed out in \cite{CondeAlonsoLoristRey2026}, the cancellative sparse domination approach developed offers a possible direction worth exploring. We hope that our counterexample sheds some light on this question.
\end{rem}

The rest of the paper is organized as follows.
In Section~\ref{Sec02}, we construct the functions $f_n$ and $g_n$.
In Section~\ref{Sec03}, we prove the estimates
\eqref{mainest01} and \eqref{mainest02}, completing the proof
of Theorem~\ref{thmmain}.
Finally, in Section~\ref{Sec04}, we explain why this construction
does not yield a counterexample to sparse domination for
the strong maximal operator.

Throughout this paper, for any two nonnegative quantities $a$ and $b$, we write $a \lesssim b$ if there exists an absolute constant $C > 0$ independent of $a$ and $b$ such that $a \le Cb$. The notation $a \simeq b$ indicates that both $a \lesssim b$ and $b \lesssim a$ hold simultaneously.

\bigskip

\section{Proof of Theorem \ref{thmmain}: Construction of $f_n$ and $g_n$} \label{Sec02}

Fix a sufficiently large integer $n$. We begin by briefly outlining the construction of $f_n$ and $g_n$ as follows. 

\begin{enumerate}
\item Construct two finite lattices $\textbf{Int}(n), \textbf{Out}(n) \subseteq [0, 1]^2$.

\item Place a nonnegative smooth bump near each point of these sets.

\item Form $f_n$ by assigning random signs to the bumps centered at $\textbf{Int}(n)$, choosing the signs to maximize
the integral of $|\HH f_n|$ against the sum of the bumps centered at $\textbf{Out}(n)$.

\item Construct $g_n$ from the bumps centered at $\textbf{Out}(n)$, with its sign chosen to agree with that of $\HH f_n$.
\end{enumerate}

Now we turn to some details. 

\subsection{Step I: Properties for the  lattices $\textbf{Int}(n)$ and $\textbf{Out}(n)$}\label{20260928subsec01}

Set $\lambda=3+2\sqrt2$, and define the lattices
$$
\textbf{Int}(n)
=\left\{\frac{ 2\left(a+b\sqrt2, \, a-b\sqrt2 \right)}{\lambda^n}:a,b\in\Z \right\}
  \cap \left[\frac{1}{8}, \frac{7}{8} \right]^2, 
$$
and
$$
\textbf{Out}(n)
=\left\{\frac{\left(2a+1+2b\sqrt2,\,2a+1-2b\sqrt2 \right)}{\lambda^n}:a,b\in\Z\right\}
  \cap \left[\frac{1}{4}, \frac{1}{2} \right]^2. 
$$
Note that both $\textbf{Int}(n)$ and $\textbf{Out}(n)$ are finite. The following lemmas are \emph{crucial} in the proof of Theorem \ref{thmmain}. More precisely, the first and fourth will yield the lower bound for
$\langle\HH f_n,g_n\rangle$, while the second and third will be used
to bound the sparse form.

\begin{lem}\label{20260928lem01}
There is an absolute constant $c_0>0$ such that
\begin{equation}\label{20260929eq01}
\#\textnormal{\textbf{Out}}(n)\ge c_0\lambda^{2n}.
\end{equation}
\end{lem}

\begin{proof}
Let $z=(z_1,z_2)\in\textnormal{\textbf{Out}}(n)$. Then there are
$a,b\in\Z$ such that
$$
z_1=\frac{2a+1+2b\sqrt2}{\lambda^n},
\qquad \textrm{and} \qquad 
z_2=\frac{2a+1-2b\sqrt2}{\lambda^n}.
$$
Now define
$$
{\bf M}:=\left\{(a,b)\in\Z^2:
\left|\frac{2a}{\lambda^n}-\frac38\right|\le\frac1{32},
\quad
\left|\frac{2b\sqrt2}{\lambda^n}\right|\le\frac1{32}
\right\}.
$$
We claim that the map $\mathcal T:{\bf M}\to\textnormal{\textbf{Out}}(n)$
given by
$$
\mathcal T(a,b)
:=(a,b){\bf T}
 +\left(\frac1{\lambda^n},\frac1{\lambda^n}\right),
\qquad
{\bf T}:=
\begin{bmatrix}
\frac2{\lambda^n} & \frac2{\lambda^n}\\
\frac{2\sqrt2}{\lambda^n} & -\frac{2\sqrt2}{\lambda^n}
\end{bmatrix},
$$
is injective. Indeed, injectivity follows from the invertibility of
${\bf T}$, and hence it remains to check that
$\mathcal T({\bf M})\subseteq\textnormal{\textbf{Out}}(n)$.
For $(a,b)\in{\bf M}$, write $z=\mathcal T(a,b)$. Then, for $i=1,2$,
$$
\begin{aligned}
\left|z_i-\frac38\right|
&=\left|\frac{2a+1\pm2b\sqrt2}{\lambda^n}-\frac38\right|\\
&\le
\left|\frac{2a}{\lambda^n}-\frac38\right|
+\left|\frac{2b\sqrt2}{\lambda^n}\right|
+\frac1{\lambda^n}\\
&\le\frac1{16}+\frac1{\lambda^n}<\frac18
\end{aligned}
$$
for sufficiently large $n$. Thus $z\in[1/4,1/2]^2$, proving the claim. As a consequence, we have 
$\#\textnormal{\textbf{Out}}(n)\ge\#{\bf M}
\gtrsim\lambda^{2n}$, which gives \eqref{20260929eq01}.
\end{proof}

\begin{lem}\label{20260928lem02}
Let ${\bf u}=(u_1,u_2)$ and ${\bf v}=(v_1,v_2)$ be distinct points belonging
to the same lattice $\textnormal{\textbf{Int}}(n)$ or
$\textnormal{\textbf{Out}}(n)$. Then
\begin{equation}\label{20260929eq04}
|u_1-v_1|\,|u_2-v_2|\ge\frac{4}{\lambda^{2n}}.
\end{equation}
Consequently, for every rectangle $R \subseteq \R^2$,
\begin{equation}\label{20260929eq05}
\#\big(\textnormal{\textbf{Int}}(n)\cap R\big), \quad 
\#\big(\textnormal{\textbf{Out}}(n)\cap R\big) \le 1+\frac{\lambda^{2n}|R|}{4}.
\end{equation}
\end{lem}

\begin{proof}
We first show \eqref{20260929eq04}. By the definitions of $\textnormal{\textbf{Int}}(n)$ and
$\textnormal{\textbf{Out}}(n)$, there are integers $a,b$,
not both zero, such that
$$
u_1-v_1=\frac{2(a+b\sqrt2)}{\lambda^n},
\qquad
u_2-v_2=\frac{2(a-b\sqrt2)}{\lambda^n}.
$$
Since $\sqrt2$ is irrational, $a^2-2b^2$ is a nonzero integer.
Therefore,
$$
|u_1-v_1|\,|u_2-v_2|
=\frac{4|a^2-2b^2|}{\lambda^{2n}}
\ge\frac{4}{\lambda^{2n}},
$$
which proves \eqref{20260929eq04}.

\vspace{0.1cm}

Next we prove \eqref{20260929eq05}. Let $R$ be a rectangle of width $w$ and height $h$.
We first prove the estimate for
$\textnormal{\textbf{Int}}(n)$ in \eqref{20260929eq05}.
Write
$$
m:=\#\big(\textnormal{\textbf{Int}}(n)\cap R\big).
$$
Without loss of generality, we may assume $m \ge 2$, otherwise there is nothing to prove. By \eqref{20260929eq04}, if ${\bf u}, {\bf v} \in \textnormal{\textbf{Int}}(n)$, then $u_1 \neq v_1$. Thus we may order the points in
$\textnormal{\textbf{Int}}(n)\cap R$ as
${\bf u}^{(1)},\ldots,{\bf u}^{(m)}$, with
$$
u^{(1)}_1<u^{(2)}_1<\cdots<u^{(m)}_1.
$$
For $j=1,\ldots,m-1$, both $u^{(j)}$ and $u^{(j+1)}$ belong to
$R$, so $|u^{(j+1)}_2-u^{(j)}_2|\le h$. Hence, by 
\eqref{20260929eq04}, we have 
$$
\big(u^{(j+1)}_1-u^{(j)}_1\big)h
\ge
\big(u^{(j+1)}_1-u^{(j)}_1\big)
\big|u^{(j+1)}_2-u^{(j)}_2\big|
\ge\frac{4}{\lambda^{2n}}.
$$
Summing over $j$, we obtain
$$
(m-1)\cdot \frac{4}{\lambda^{2n}h}
\le
\sum_{j=1}^{m-1}\big(u^{(j+1)}_1-u^{(j)}_1\big)
=u^{(m)}_1-u^{(1)}_1
\le w.
$$
Since $|R|=wh$, this proves the first estimate in
\eqref{20260929eq05}. The same argument applies to
$\textnormal{\textbf{Out}}(n)$.
\end{proof}

\begin{lem}\label{20260928lem03}
Let $p=(p_1,p_2)\in\textnormal{\textbf{Int}}(n)$ and
$z=(z_1,z_2)\in\textnormal{\textbf{Out}}(n)$. Then
\begin{equation}\label{eq:cross}
|p_1-z_1|\,|p_2-z_2|\ge\frac{1}{\lambda^{2n}}.
\end{equation}
\end{lem}

\begin{proof}
By the definitions of $\textnormal{\textbf{Int}}(n)$ and
$\textnormal{\textbf{Out}}(n)$, there exist $a,b \in \Z$ such that
$$
p_1-z_1=\frac{2a-1+2b\sqrt2}{\lambda^n},
\qquad
p_2-z_2=\frac{2a-1-2b\sqrt2}{\lambda^n}.
$$
Consequently,
$$
|p_1-z_1|\,|p_2-z_2|
=\frac{\big|(2a-1)^2-8b^2\big|}{\lambda^{2n}}
\ge\frac{1}{\lambda^{2n}},
$$
since $(2a-1)^2-8b^2$ is an odd integer. The proof is complete. 
\end{proof}

\begin{lem}\label{20261001lem04}
Let $z=(z_1,z_2)\in\textnormal{\textbf{Out}}(n)$, and define
\begin{equation}\label{20261001eq01}
p_{n,z,j}:=z+\left(\lambda^{j-n},\lambda^{-j-n}\right),
\qquad j=1,\ldots,n-2.
\end{equation}
Then $p_{n, z, j} \in \textnormal{\textbf{Int}}(n)$. 
\end{lem}

\begin{proof}
Recall $\lambda=3+2\sqrt{2}$. Hence, for each integer $j\ge1$, one has
$$
\lambda^j=A_j+B_j\sqrt2,
\qquad \textrm{for some} \quad A_j,B_j\in\Z.
$$
It is clear that $(A_1, B_1)=(3, 2)$. Moreover, we have the following elementary observations: for each $j \ge 1$, 
\begin{enumerate}
\item  $A_{j+1}=3A_j+4B_j$ and $B_{j+1}=2A_j+3B_j$;
    \item $\lambda^{-j}=(3-2\sqrt{2})^j=A_j-B_j\sqrt{2}$;
    \item $A_j$ is odd and $B_j$ is even.
\end{enumerate}

\vspace{0.1cm}

\noindent Now take any $z\in\textnormal{\textbf{Out}}(n)$, and write
$$
z=\frac{(2a+1+2b\sqrt2,\,2a+1-2b\sqrt2)}{\lambda^n}.
$$
for some $a, b \in \Z$. Therefore, for any $j=1,\ldots,n-2$, we have 
$$
p_{n,z,j}=\frac{\big(2a+1+A_j+(2b+B_j)\sqrt2,\,2a+1+A_j-(2b+B_j)\sqrt2\big)
}{\lambda^n}.
$$
Since $A_j$ is odd and $B_j$ is even, we have 
$$
a':=a+\frac{A_j+1}{2},
\qquad
b':=b+\frac{B_j}{2}
$$
are integers, which gives 
$$
p_{n,z,j}=\frac{2(a'+b'\sqrt2,\,a'-b'\sqrt2)}{\lambda^n}.
$$
Hence, to show $p_{n, z, j} \in \textnormal{\textbf{Int}}(n)$, it remains to verify that $p_{n,z,j}\in[1/8,7/8]^2$. Indeed, for $1\le j\le n-2$, it is easy to see that 
$$
0<\lambda^{j-n}\le\lambda^{-2}<\frac1{16},
\qquad \textrm{and} \qquad 
0<\lambda^{-j-n}\le\lambda^{-n-1}<\frac1{16}, 
$$
which, together with the fact that $z\in[1/4,1/2]^2$, gives
$p_{n,z,j}\in[1/4,9/16]^2\subset[1/8,7/8]^2$. The proof is complete. 
\end{proof}

\subsection{Step II: Construction of the bumps}

Fix a nonnegative function
$\phi\in C_c^\infty((-1/10,1/10)^2)$ satisfying $\int_{\R^2}\phi(x)\,dx=1$.
For each
$u\in\textnormal{\textbf{Int}}(n)\cup
\textnormal{\textbf{Out}}(n)$, define the bump 
$$
\psi_{n,u}(x)
:=\lambda^{2n}\phi\big(\lambda^{2n}(x-u)\big).
$$
Observe that $\left\| \psi_{n, u}\right\|_{L^1(\R^2)}=\lambda^{-2n}$. Moreover,
$\supp \; \psi_{n,u}
\subseteq u+\left[-\frac{1}{10\lambda^{2n}},
\frac{1}{10\lambda^{2n}}\right]^2$. 
Since all the centers belong to $[1/8,7/8]^2$, we conclude that each $\psi_{n,u}$ is smooth and compactly supported in $(0,1)^2$.

\vspace{0.1cm}

Now define
$$
\mu_n:=\sum_{p\in\textnormal{\textbf{Int}}(n)}\psi_{n,p},
\qquad
\nu_n:=\sum_{z\in\textnormal{\textbf{Out}}(n)}\psi_{n,z}.
$$
Note that both sums above are finite, and hence $\mu_n$ and $\nu_n$ are nonnegative
functions in $C_c^\infty((0,1)^2)$. By
Lemma~\ref{20260928lem01}, we also have
\begin{equation}\label{20261001eq02}
\int_{\R^2}\nu_n(x)\,dx
=\frac{\#\textnormal{\textbf{Out}}(n)}{\lambda^{2n}}
\ge c_0.
\end{equation}

\subsection{Step III: Construction of $f_n$}

Let $\varepsilon=(\varepsilon_p)_{p\in\textnormal{\textbf{Int}}(n)}$
be a family of independent Rademacher random variables, and define
$$
f_{n,\varepsilon}
:=\sum_{p\in\textnormal{\textbf{Int}}(n)}
\varepsilon_p\psi_{n,p}.
$$
For every choice of signs, it is clear that 
$f_{n,\varepsilon}\in C_c^\infty((0,1)^2)$ and
$|f_{n,\varepsilon}(x)|\le\mu_n(x), \; x\in\R^2$. 

\vspace{0.1cm}

We now choose the signs to maximize the quantity
\begin{equation}\label{20261001eq03}
J_n(\varepsilon)
:=\int_{\R^2}|\HH f_{n,\varepsilon}(x)|\nu_n(x)\,dx.
\end{equation}
We first show that, for each fixed $n$, $J_n$ is bounded uniformly
over all choices of signs. By \eqref{BPHTeq01}, we have
$$
\begin{aligned}
J_n(\varepsilon)
&=\int_{\supp \; \nu_n}
|\HH f_{n,\varepsilon}(x)|\nu_n(x)\,dx\\
&\lesssim 
\int_{\supp \; \nu_n}\int_{\supp \; \mu_n}
\frac{|f_{n,\varepsilon}(y)|\nu_n(x)}
{|x_1-y_1|\,|x_2-y_2|}\,dy\,dx\\
&\lesssim
\int_{\supp \; \nu_n}\int_{\supp \; \mu_n}
\frac{\mu_n(y)\nu_n(x)}
{|x_1-y_1|\,|x_2-y_2|}\,dy\,dx.
\end{aligned}
$$
To estimate the last integral, take $x\in \supp \; \nu_n$ and
$y\in \supp \; \mu_n$. There are
$z\in\textnormal{\textbf{Out}}(n)$ and
$p\in\textnormal{\textbf{Int}}(n)$ such that
$x\in\supp \; \psi_{n,z}$ and $y\in\supp \; \psi_{n,p}$.
Since $p,z\in[0,1]^2$, by Lemma~\ref{20260928lem03}, we have 
$$
|p_i-z_i|
\ge |p_1-z_1|\,|p_2-z_2|
\ge\frac{1}{\lambda^{2n}},
\qquad i=1,2.
$$
Thus, by the support condition on the bumps,
\begin{align} \label{20261001eq05X}
|x_i-y_i|
&\ge |z_i-p_i|-|x_i-z_i|-|y_i-p_i| \nonumber \\
&\ge\frac{1}{\lambda^{2n}}
-\frac{1}{10\lambda^{2n}}
-\frac{1}{10\lambda^{2n}}
=\frac{4}{5\lambda^{2n}},
\qquad i=1,2.
\end{align}

Substituting these estimates into the preceding integral, we obtain
\begin{align} \label{20261001eq05}
J_n(\varepsilon)
&\lesssim \lambda^{4n} 
\left(\int_{\R^2}\mu_n(y)\,dy\right)
\left(\int_{\R^2}\nu_n(x)\,dx\right) \nonumber \\
&=\#\textnormal{\textbf{Int}}(n)\,
\#\textnormal{\textbf{Out}}(n)
<\infty.
\end{align}

\vspace{0.1cm}

We also need the smoothness of $\HH f_{n,\varepsilon}$ near
$\supp \; \nu_n$. Indeed, by the argument in \eqref{20261001eq05X}, we can take 
an open neighborhood $U_n$ of $\supp \; \nu_n$ such that
$$
|x_i-y_i|\ge\frac{1}{2\lambda^{2n}},
\qquad
x\in U_n,\quad y\in\supp \; \mu_n,\quad i=1,2.
$$
Hence, for $x \in U_n$, since $\supp \; f_{n, \varepsilon} \subseteq \supp \; \mu_n$, we see that 
\begin{align} \label{20261001eq20}
\HH f_{n,\varepsilon}(x)
&=\frac{1}{\pi^2}
\lim_{\varepsilon_1,\varepsilon_2\to0^+}
\int_{\substack{|x_1-y_1|>\varepsilon_1 \nonumber \\
|x_2-y_2|>\varepsilon_2}}
\frac{f_{n,\varepsilon}(y)}
{(x_1-y_1)(x_2-y_2)}\,dy\\
&=\frac{1}{\pi^2}\int_{\R^2}
\frac{f_{n,\varepsilon}(y)}
{(x_1-y_1)(x_2-y_2)}\,dy.
\end{align}

Arguing as in \eqref{20261001eq05}, we see that the above integral
is absolutely convergent. Moreover, a similar reasoning also shows \eqref{20261001eq20} is differentiable on $U_n$ up to every order. Hence $\HH f_{n,\varepsilon}$ is smooth on $U_n$.

\vspace{0.1cm}

Finally, since $\textnormal{\textbf{Int}}(n)$ is finite, there are only
finitely many choices of signs. We may therefore choose
$\varepsilon^{(n)}$ such that
$$
J_n(\varepsilon^{(n)})
=\max_{\varepsilon}J_n(\varepsilon),
$$
and define
\begin{equation}\label{20260928inputfcn}
f_n:=f_{n,\varepsilon^{(n)}}
=\sum_{p\in\textnormal{\textbf{Int}}(n)}
\varepsilon_p^{(n)}\psi_{n,p}.
\end{equation}
This completes the construction of $f_n$. 

\subsection{Step IV: Construction of $g_n$}

Let
\begin{equation} \label{20261002eq01}
h_n:=\HH f_n\quad\textrm{on }U_n.
\end{equation} 
By \eqref{20261001eq20}, $h_n$ is a smooth function on $U_n$. Define
$$
g_n(x):=
\begin{cases}
\displaystyle
\nu_n(x)\frac{h_n(x)}{\sqrt{1+h_n(x)^2}},
& x\in U_n,\\[6pt]
0, & x\notin U_n.
\end{cases}
$$
Observe that $g_n$ is smooth on $U_n$ and satisfies $|g_n(x)|\le\nu_n(x)$. Moreover, 
$g_n\in C_c^\infty((0,1)^2)$.

\vspace{0.1cm}

In summary, we have now constructed real-valued functions $f_n$ and $g_n$
satisfying
\begin{equation}\label{20261001eq30}
f_n,g_n\in C_c^\infty((0,1)^2),
\qquad
|f_n|\le\mu_n,
\qquad
|g_n|\le\nu_n.
\end{equation}
It remains to prove the two estimates in Theorem~\ref{thmmain}.

\medskip 

\section{Proof of Theorem~\ref{thmmain}: the estimates \eqref{mainest01} and \eqref{mainest02}}\label{Sec03}

In this section, we show that the functions $f_n$ and $g_n$
constructed in the preceding section satisfy the main estimates
\eqref{mainest01} and \eqref{mainest02}.

\subsection{The lower bound for $\langle\HH f_n,g_n\rangle$}
\label{20261003subseclower}

Fix $z\in\textnormal{\textbf{Out}}(n)$, and recall that by Lemma~\ref{20261001lem04}, the points
$$
p_{n,z,j}
=z+\left(\lambda^{j-n},\lambda^{-j-n}\right),
\qquad j=1,\ldots,n-2,
$$
belong to $\textnormal{\textbf{Int}}(n)$.
Note also that these points are distinct since their first coordinates are strictly
increasing in $j$. Let $x\in\supp \; \psi_{n,z}$ and
$y\in \; \supp \; \psi_{n,p_{n,z,j}}$, and by definition, we have 
$$
|x_i-z_i|\le\frac1{10\lambda^{2n}},
\qquad
\left|y_i-\left(p_{n,z,j}\right)_i\right|
\le\frac1{10\lambda^{2n}},
\qquad i=1,2.
$$
Consequently,
$$
\begin{aligned}
\left|(y_1-x_1)-\lambda^{j-n}\right|
&=\left|y_1-x_1- \left(p_{n, z, j} \right)_1+z_1\right| \\
&\le
\left|y_1-\left(p_{n,z,j}\right)_1\right|
+|x_1-z_1|\\
&\le\frac1{5\lambda^{2n}}
\le\frac15\lambda^{j-n},
\end{aligned}
$$
and similarly,
$$
\begin{aligned}
\left|(y_2-x_2)-\lambda^{-j-n}\right|
&=\left|y_2-x_2- \left(p_{n, z, j} \right)_2+z_2 \right| \\ 
&\le
\left|y_2-\left(p_{n,z,j}\right)_2\right|
+|x_2-z_2|\\
&\le\frac1{5\lambda^{2n}}
\le\frac15\lambda^{-j-n}.
\end{aligned}
$$
Thus,
$$
\frac45\lambda^{j-n}
\le y_1-x_1
\le\frac65\lambda^{j-n},
\qquad
\frac45\lambda^{-j-n}
\le y_2-x_2
\le\frac65\lambda^{-j-n},
$$
which gives
$$
0<(x_1-y_1)(x_2-y_2)
\lesssim \frac{1}{\lambda^{2n}}.
$$
Since $\psi_{n,p_{n,z,j}}$ is nonnegative and has $L^1$-norm 
$\lambda^{-2n}$, we obtain
\begin{align} \label{20261001eq23}
\HH\psi_{n,p_{n,z,j}}(x)
&=\frac1{\pi^2}
\int_{\R^2} \frac{\psi_{n,p_{n,z,j}}(y)}
{(x_1-y_1)(x_2-y_2)}\,dy \nonumber \\
&\gtrsim \lambda^{2n} \cdot 
\int_{\R^2}\psi_{n,p_{n,z,j}}(y)\,dy=1,
\end{align}
for any $x\in\supp \; \psi_{n,z}$ and $j=1,\ldots,n-2$.

\vspace{0.1cm}

Now take any $x\in\supp \; \nu_n$, there exists some $z \in \textnormal{\textbf{Out}}(n)$ such that $x \in \supp \; \psi_{n, z}$. By Khintchine's inequality
and \eqref{20261001eq23}, we obtain
\begin{align} \label{20261001eq26}
\mathbb E|\HH f_{n,\varepsilon}(x)|
&=\mathbb E\left|
\sum_{p\in\textnormal{\textbf{Int}}(n)}
\varepsilon_p\HH\psi_{n,p}(x)
\right| \gtrsim \left(
\sum_{p\in\textnormal{\textbf{Int}}(n)}
|\HH\psi_{n,p}(x)|^2
\right)^{1/2} \nonumber \\
&\ge
\left(\sum_{j=1}^{n-2}
|\HH\psi_{n,p_{n,z,j}}(x)|^2
\right)^{1/2} \gtrsim \sqrt{n-2}.
\end{align} 
Since $\varepsilon^{(n)}$ was chosen to maximize
$J_n(\varepsilon)$, by \eqref{20261001eq26}, we deduce that 
\begin{align} \label{20261001eq35}
\int_{\R^2}|\HH f_n(x)|\nu_n(x)\,dx
&=J_n(\varepsilon^{(n)})=\max_{\varepsilon}J_n(\varepsilon) \nonumber \\
&\ge\mathbb E J_n(\varepsilon)=\int_{\R^2}
\mathbb E|\HH f_{n,\varepsilon}(x)|\nu_n(x)\,dx \nonumber  \\
&\gtrsim \sqrt{n-2} \cdot 
\int_{\R^2}\nu_n(x)\,dx.
\end{align}
To this end, we record the elementary estimate
\begin{equation}\label{20261001eq37}
\frac{s^2}{\sqrt{1+s^2}}
\ge\frac{s^2}{1+|s|}
=|s|-1+\frac1{1+|s|}
\ge|s|-1,
\qquad s\in\R.
\end{equation}
We are now ready to establish a lower bound for $\langle \HH f_n, g_n \rangle$. Recall that by \eqref{20261002eq01}, $h_n=\HH f_n$ on $U_n$. Hence, by the definition of $g_n$, \eqref{20261001eq37}, \eqref{20261001eq35} and \eqref{20261001eq02},  we have
\begin{align*}
\langle\HH f_n,g_n\rangle
&=\int_{U_n}
\nu_n(x)\frac{h_n(x)^2}{\sqrt{1+h_n(x)^2}}\,dx\\
&\ge\int_{U_n}|h_n(x)|\nu_n(x)\,dx
-\int_{U_n}\nu_n(x)\,dx\\
&=\int_{\R^2}|\HH f_n(x)|\nu_n(x)\,dx
-\int_{\R^2}\nu_n(x)\,dx\\
&\gtrsim 
\left(\sqrt{n-2}-1\right)
\int_{\R^2}\nu_n(x)\,dx \gtrsim \sqrt{n} 
\end{align*}
for $n \ge 1$ sufficiently large. The proof for \eqref{mainest01} is complete. 

\subsection{The upper bound for $\Lambda_{\calS}(f_n,g_n)$}
\label{20261003subsecupper}

By \eqref{20261001eq30},
$|f_n|\le\mu_n$ and $|g_n|\le\nu_n$. Hence,
\begin{equation} \label{20261002eq30}
\Lambda_{\calS}(f_n,g_n)
\le\sum_{R\in\calS}|R|
\langle\mu_n\rangle_R\langle\nu_n\rangle_R,
\end{equation}
where we recall $\calS$ is a sparse family of rectangles in $\R^2$. To estimate the averages on the right-hand side of \eqref{20261002eq30}, we first
prove the following lemma.

\begin{lem}\label{20261002lem01}
Let $R\subseteq\R^2$ be a rectangle such that
\begin{equation}\label{20261002eq10}
\left(\int_R\mu_n\right)\left(\int_R\nu_n\right)>0.
\end{equation}
Then
\begin{equation}\label{20261002claim01}
\int_R\mu_n\le5|R|,
\qquad
\int_R\nu_n\le5|R|.
\end{equation}
\end{lem}

\begin{proof}
The two estimates in \eqref{20261002claim01} are proved in the same way, so we only prove
the first one. Write $R=[a_1,b_1)\times[a_2,b_2)$ and let
$\ell_i=b_i-a_i$, $i=1,2$ be the side--lengths of $R$. We have the following observations. 

\begin{enumerate} 
\item 
\begin{equation} \label{20261002obs01}
\ell_i\ge 1/(2\lambda^{2n}), \qquad i=1,2.
\end{equation} 
\item 
\begin{equation} \label{20261002obs02}
|R|\ge 1/(4\lambda^{2n}).
\end{equation}
\end{enumerate}
To see these observations, we first note that by \eqref{20261002eq10}, there exists 
$p\in\textnormal{\textbf{Int}}(n)$ and
$z\in\textnormal{\textbf{Out}}(n)$ such that
$$
\int_R\psi_{n,p}>0,
\qquad \textrm{and} \qquad 
\int_R\psi_{n,z}>0.
$$
Choose $x, y\in R$ with $\psi_{n,z}(x)>0$ and $\psi_{n,p}(y)>0$, and hence 
$$
|y_i-p_i|\le\frac1{10\lambda^{2n}},
\qquad
|x_i-z_i|\le\frac1{10\lambda^{2n}},
\qquad i=1,2.
$$
By Lemma~\ref{20260928lem03}, we have 
$$
|p_1-z_1|\,|p_2-z_2|\ge\frac1{\lambda^{2n}}.
$$
Since $p,z\in[0,1]^2$, $|p_i-z_i| \le 1, \; i=1, 2$. It follows that $|p_i-z_i|\ge\lambda^{-2n}$ for $i=1,2$.
Thus,
$$
\begin{aligned}
\ell_i
&\ge|x_i-y_i|\\
&\ge|p_i-z_i|-|y_i-p_i|-|x_i-z_i|\\
&\ge|p_i-z_i|-\frac1{5\lambda^{2n}}\\
&
\ge\frac12|p_i-z_i| \ge \frac1{2\lambda^{2n}},
\end{aligned}
$$
which proves \eqref{20261002obs01}. Also,
$$
|R|=\ell_1\ell_2
\ge\frac14|p_1-z_1|\,|p_2-z_2|
\ge\frac1{4\lambda^{2n}},
$$
which proves \eqref{20261002obs02}. 

\vspace{0.1cm}

We now count the centers of the bumps whose integrals over
$R$ are positive. Define
$$
R^*:=
\left[a_1-\frac1{10\lambda^{2n}},
      b_1+\frac1{10\lambda^{2n}}\right)
\times
\left[a_2-\frac1{10\lambda^{2n}},
      b_2+\frac1{10\lambda^{2n}}\right).
$$
By the support condition on the bumps, if
$\int_R\psi_{n,p}>0$, then $p\in R^*$.
By \eqref{20261002obs01}, the side lengths of $R^*$ satisfy
$$
\ell_i+\frac1{5\lambda^{2n}}\le2\ell_i,
\qquad i=1,2,
$$
and hence $|R^*|\le4|R|$. Finally, by $\|\psi_{n,p}\|_{L^1(\R^2)}=\lambda^{-2n}$,
\eqref{20260929eq05}, and \eqref{20261002obs02}, we have 
$$
\begin{aligned}
\int_R\mu_n(x)\,dx
&=\sum_{p\in\textnormal{\textbf{Int}}(n)\cap R^*}
  \int_R\psi_{n,p}(x)\,dx\\
&\le\frac1{\lambda^{2n}} \cdot 
  \#\big(\textnormal{\textbf{Int}}(n)\cap R^*\big)\\
&\le\frac1{\lambda^{2n}}+\frac{|R^*|}{4}\\
&\le\frac1{\lambda^{2n}}+|R|\\
&\le 5|R|.
\end{aligned}
$$
The proof is complete. 
\end{proof}

We now return to the estimate for $\Lambda_{\calS}(f_n,g_n)$. It is clear that 
$$
\int_R\mu_n=\int_{R\cap[0,1)^2}\mu_n
\qquad \textrm{and} \qquad 
\int_R\nu_n=\int_{R\cap[0,1)^2}\nu_n.
$$
If both integrals are positive, then $R\cap[0,1)^2$ is a
rectangle, and Lemma~\ref{20261002lem01} gives
$$
\int_R\mu_n\le5|R\cap[0,1)^2|
\qquad \textrm{and} \qquad 
\int_R\nu_n\le5|R\cap[0,1)^2|.
$$
Therefore, for every rectangle $R$,
\begin{equation}\label{20261002claim01X}
\langle\mu_n\rangle_R\langle\nu_n\rangle_R
\lesssim \left(\frac{|R\cap[0,1)^2|}{|R|}\right)^2.
\end{equation}

To sum \eqref{20261002claim01X} over $\calS$, we use
the following lemma.

\begin{lem}\label{20261002lem02}
Define
$$
m(s):=\frac1{1+\operatorname{dist}(s,[0,1])},
\qquad s\in\R,
$$
and
$$
W(x):=m(x_1)m(x_2),
\qquad x\in\R^2.
$$
Then, for every rectangle $R\subseteq\R^2$,
$$
\frac{|R\cap[0,1)^2|}{|R|}\le W(x),
\qquad x\in R.
$$
\end{lem}

\begin{proof}
It is clear that  
$$
m(s)=
\begin{cases}
\dfrac1{1-s}, & s<0,\\[4pt]
1, & 0\le s\le1,\\[4pt]
\dfrac1s, & s>1.
\end{cases}
$$
We first show that, for every bounded interval $I$ of positive
length and every $s\in I$,
\begin{equation}\label{20261002eq23}
\frac{|I\cap[0,1]|}{|I|}\le m(s).
\end{equation}
It suffices to consider the case when $s\notin[0,1]$ with
$|I\cap[0,1]|>0$; otherwise \eqref{20261002eq23} is trivial. Since $I$ contains both $s$ and
$I\cap[0,1]$, $|I|\ge|I\cap[0,1]|+\operatorname{dist}(s,[0,1])$. Thus, 
$$
\begin{aligned}
\frac{|I\cap[0,1]|}{|I|}
&\le
\frac{|I\cap[0,1]|}
{|I\cap[0,1]|+\operatorname{dist}(s,[0,1])}\\
&\le\frac1{1+\operatorname{dist}(s,[0,1])}
=m(s),
\end{aligned}
$$
which proves \eqref{20261002eq23}.

\vspace{0.1cm}

Now write $R=I_1\times I_2$. For every $x=(x_1,x_2)\in R$,
applying \eqref{20261002eq23} to each interval, we have
$$
\frac{|R\cap[0,1)^2|}{|R|}
=\frac{|I_1\cap[0,1]|}{|I_1|} \cdot 
 \frac{|I_2\cap[0,1]|}{|I_2|}
\le m(x_1) \cdot m(x_2)
=W(x)
$$
as desired. 
\end{proof}

Let $\calS$ be any sparse family of rectangles in $\R^2$, and choose
pairwise disjoint measurable sets $E_R\subset R$ satisfying
$|E_R|\ge|R|/2$. By \eqref{20261002eq30}, and 
\eqref{20261002claim01X}, we see that 
$$
\begin{aligned}
\Lambda_{\calS}(f_n,g_n)
&\le\sum_{R\in\calS}|R|
\langle\mu_n\rangle_R\langle\nu_n\rangle_R\\
&\lesssim \sum_{R\in\calS}|R|
\left(\frac{|R\cap[0,1)^2|}{|R|}\right)^2\\
&\lesssim  \sum_{R\in\calS}|E_R|
\left(\frac{|R\cap[0,1)^2|}{|R|}\right)^2\\
&\lesssim \sum_{R\in\calS}\int_{E_R}W(x)^2\,dx\\
&\lesssim \int_{\R^2}W(x)^2\,dx \lesssim 1,
\end{aligned}
$$
where in the second last inequality, we used the disjointness of the sets
$E_R$. This proves \eqref{mainest02}.

\medskip 

Finally, combining \eqref{mainest01} and \eqref{mainest02}, the proof of
Theorem~\ref{thmmain} is complete. 

\section{A further remark on the strong maximal operator} \label{Sec04}

Recall that the strong maximal operator is defined by
$$
{\bf M}f(x)
:=\sup_{R\ni x}\frac1{|R|}\int_R|f(y)|\,dy,
$$
where the supremum is taken over all rectangles containing $x$. In the final part of this paper, we show that the functions $f_n$ and $g_n$ constructed above do \emph{not} give
a counterexample to sparse domination for ${\bf M}$.
More precisely, let
$$
\calS_0:=\big\{[0,1)^2\big\}, 
$$
and we will show that
\begin{equation}\label{20261002eq40}
\big|\langle {\bf M}f_n,g_n\rangle\big|
\lesssim \frac{1}{c_0}\Lambda_{\calS_0}(f_n,g_n),
\end{equation}
where $c_0$ is the constant in Lemma~\ref{20260928lem01}.

\vspace{0.1cm}

We first observe that
\begin{equation}\label{20261002eq41}
{\bf M}\mu_n(x)\lesssim 1
\qquad\textrm{whenever }\nu_n(x)>0.
\end{equation}
Indeed, fix such a point $x$ and let $R$ be any rectangle
containing $x$, and hence, by the continuity of $\nu_n$, we have $\int_R \nu_n>0$.  If $\int_R\mu_n>0$, by
Lemma~\ref{20261002lem01}, we have 
$$
\frac1{|R|}\int_R\mu_n \lesssim 1. 
$$
The same inequality holds if $\int_R\mu_n=0$.
Taking the supremum over all rectangles containing $x$
proves \eqref{20261002eq41}. Since $|f_n|\le\mu_n$ and
$|g_n|\le\nu_n$, it follows that
\begin{align} \label{20261002eq42}
\big|\langle {\bf M}f_n,g_n\rangle\big|
&\le\int_{\R^2}{\bf M}f_n(x)|g_n(x)|\,dx \nonumber \\
&\le\int_{\{x:\nu_n(x)>0\}}
{\bf M}\mu_n(x)|g_n(x)|\,dx \nonumber \\
&\lesssim \|g_n\|_{L^1(\R^2)}.
\end{align}

We next show that
\begin{equation}\label{20261002eq43}
\|f_n\|_{L^1(\R^2)}\ge c_0.
\end{equation}
For this, we first claim that for distinct $p,q\in\textnormal{\textbf{Int}}(n)$, one has
\begin{equation} \label{20261002claim03}
\supp \; \psi_{n,p}\cap\supp \; \psi_{n,q}=\emptyset.
\end{equation} 
Assume the contrary, then by the support condition on the bumps, we have
$$
|p_i-q_i|\le\frac1{5\lambda^{2n}},
\qquad i=1,2,
$$
and hence
$$
|p_1-q_1|\,|p_2-q_2|
\le\frac1{25\lambda^{4n}}
<\frac4{\lambda^{2n}},
$$
which is a contradiction to Lemma~\ref{20260928lem02}. Thus, \eqref{20261002claim03} holds. As a consequence, since the coefficients in
\eqref{20260928inputfcn} belong to $\{-1,1\}$, we obtain
\begin{equation} \label{20261002eq45}
|f_n(x)|=\sum_{p\in\textnormal{\textbf{Int}}(n)}\psi_{n,p}(x)=\mu_n(x),
\qquad x\in\R^2.
\end{equation} 

On the other hand, by Lemma~\ref{20261001lem04}, the map
$$
z\longmapsto p_{n,z,1}
=z+\left(\lambda^{1-n},\lambda^{-1-n}\right)
$$
sends $\textnormal{\textbf{Out}}(n)$ into
$\textnormal{\textbf{Int}}(n)$. Observe that this map is injective, and therefore
\begin{equation} \label{20261002eq46}
\#\textnormal{\textbf{Int}}(n)
\ge\#\textnormal{\textbf{Out}}(n).
\end{equation} 
Hence, by \eqref{20261002eq45}, the fact that each bump has $L^1(\R^2)$ norm $\lambda^{-2n}$, \eqref{20261002eq46}, and \eqref{20261001eq02}, we see that 
$$
\begin{aligned}
\|f_n\|_{L^1(\R^2)}
&=\int_{\R^2}\mu_n(x)\,dx\\
&=\frac{\#\textnormal{\textbf{Int}}(n)}{\lambda^{2n}}\\
&\ge\frac{\#\textnormal{\textbf{Out}}(n)}{\lambda^{2n}}
=\int_{\R^2}\nu_n(x)\,dx
\ge c_0,
\end{aligned}
$$
which proves \eqref{20261002eq43}.

\vspace{0.1cm}

Finally, since $f_n$ and $g_n$ are supported in $(0,1)^2$,
$$
\Lambda_{\calS_0}(f_n,g_n)
=\|f_n\|_{L^1(\R^2)}\|g_n\|_{L^1(\R^2)}.
$$
Combining \eqref{20261002eq42} and \eqref{20261002eq43}
gives
$$
\big|\langle {\bf M}f_n,g_n\rangle\big|
\lesssim \|g_n\|_{L^1(\R^2)}
\lesssim \frac{1}{c_0}\Lambda_{\calS_0}(f_n,g_n),
$$
as claimed.

\bigskip 

\noindent{\bf Acknowledgement.}
The author would like to thank his Ph.D. advisor, Brian Street,
for suggesting this problem in 2019 after the author had
completed his thesis problem, and for his guidance and support
throughout the author's doctoral studies at the University of Wisconsin, Madison.
The author was supported by the NSF grant DMS-2555999
and by the Simons Travel grant MPS-TSM-00007213.

\end{document}